\documentclass[12pt]{amsart}

\usepackage{times}
\usepackage{amsmath,amsfonts,amssymb,amsopn,amscd,amsthm}
\usepackage{mathbbol}
\usepackage{mathrsfs}
\usepackage{graphicx,xspace}
\usepackage{epsfig}
\usepackage{enumerate} 
\usepackage{enumitem}
\usepackage{xcolor}
\usepackage[all]{xypic}

\usepackage[T1]{fontenc}
\usepackage[utf8]{inputenc}

\usepackage{hyperref}
\usepackage{psfrag}
\usepackage{bm}
\usepackage{youngtab}
\usepackage{tikz}
\usetikzlibrary{snakes}
\usepackage[all]{xy}
\usepackage{varioref}

\theoremstyle{plain}
\newtheorem{theoreme}{Theorem}[section]
\newtheorem{prop}[theoreme]{Proposition}

\newtheorem{cor}[theoreme]{Corollary}

\theoremstyle{remark}
\newtheorem{definition}[theoreme]{Definition}
\newtheorem{ex}[theoreme]{Example}
\newtheorem{remark}[theoreme]{Remark}
\newtheorem{c-ex}[theoreme]{Counter-example}

\date{}

\DeclareUnicodeCharacter{03C7}{\chi}
\DeclareUnicodeCharacter{B2}{^{2}}
\DeclareMathOperator{\supp}{supp}

\DeclareMathOperator{\Proj}{Proj}

\DeclareMathOperator{\ct}{ct}
\DeclareMathOperator{\ty}{type}

\author[O. Tout]{Omar Tout}
\address{Instytut Matematyczny, Polska Akademia Nauk,
ul. Śniadeckich 8,
00-656 Warszawa,
Poland}
\email{otout@impan.pl}

\title[$k$-partial permutations]{$k$-partial permutations and the center of the \\wreath product $\mathcal{S}_k\wr \mathcal{S}_n$ algebra}

\thanks{This research is supported by Narodowe Centrum Nauki, grant number 2017/26/A/ST1/00189.}
 
\begin{document}
\maketitle

\begin{abstract} We generalize the concept of partial permutations of Ivanov and Kerov and introduce $k$-partial permutations. This allows us to show that the structure coefficients of the center of the wreath product $\mathcal{S}_k\wr \mathcal{S}_n$ algebra are polynomials in $n$ with non-negative integer coefficients. We use a universal algebra $\mathcal{I}_\infty^k$ which projects on the center $Z(\mathbb{C}[\mathcal{S}_k\wr \mathcal{S}_n])$ for each $n.$ We show that $\mathcal{I}_\infty^k$ is isomorphic to the algebra of shifted symmetric functions on many alphabets.
\end{abstract}

\section{Introduction}
For a positive integer $n,$ we denote by $\mathcal{S}_n$ the symmetric group on the set $[n] := \lbrace 1,2,\ldots ,n\rbrace.$ The center of the symmetric group algebra, usually denoted $Z(\mathbb{C}[\mathcal{S}_n]),$ is the algebra over $\mathbb{C}$ generated by the conjugacy classes of the symmetric group $\mathcal{S}_n.$ The cycle-type of a permutation $\omega\in \mathcal{S}_n$ is the partition of $n$ obtained from the lengths of the cycles in the decomposition of $\omega$ as product of disjoint cycles. It is well known that the conjugacy class of a permutation $x\in \mathcal{S}_n$ is the set of all permutations $y$ that have the same cycle-type as $x.$ Thus the family $(\mathbf{C}_\lambda)_\lambda$ indexed by the partitions of $n$ and defined by  
$$\mathbf{C}_\lambda:=\sum \omega$$
where the sum runs over all the permutations $\omega\in \mathcal{S}_n$ with cycle-type $\lambda$ is a linear basis for $Z(\mathbb{C}[\mathcal{S}_n]).$ The structure coefficients $c_{\lambda\delta}^{\rho}$ are the non-negative integers defined by the following product in $Z(\mathbb{C}[\mathcal{S}_n])$
$$\mathbf{C}_\lambda \mathbf{C}_\delta=\sum_{\rho \text{ partition of } n} c_{\lambda\delta}^{\rho} \mathbf{C}_\rho.$$ The basic way to compute these coefficients is to fix a permutation $z$ of cycle-type $\rho$ and count the pairs $(x,y)$ such that $x$ has cycle-type $\lambda,$ $y$ has cycle-type $\delta$ and $xy=z.$ This method was used by Katriel and Paldus in \cite{katriel1987explicit} to obtain the complete expression for the product of the class of transpositions $\mathbf{C}_{(2,1^{n-2})}$ with an arbitrary class $\mathbf{C}_\delta.$ However, this method is not appropriate when we consider more complicated classes. In \cite{FaharatHigman1959}, Farahat and Higman showed that the coefficients 
$c_{\lambda\delta}^\rho$ are polynomials in $n$ when the partitions $\lambda , \delta$ and $\rho$ are fixed partitions, completed with parts equal to $1$ to get partitions of $n.$ By introducing partial permutations in \cite{Ivanov1999}, Ivanov and Kerov gave a combinatorial proof to this result. They built a combinatorial algebra which projects onto the center of the symmetric group algebra $Z(\mathbb{C}[\mathcal{S}_n])$ for each $n.$ They showed that this algebra is isomorphic to the algebra of shifted symmetric functions.\\

In \cite{Tout2018}, we introduced a group of blocks permutations, denoted $\mathcal{B}_{kn}^k,$ which is isomorphic to the wreath product $\mathcal{S}_k\wr \mathcal{S}_n.$ We showed that its conjugacy classes can be indexed by families of partitions indexed by the partitions of $k$ and we proved, using the general framework given in \cite{Tout2017}, that the structure coefficients of $Z(\mathbb{C}[\mathcal{B}_{kn}^k])$ are polynomials in $n$ under certain conditions. When $k=1,$ $\mathcal{B}_{n}^1$ is the symmetric group $\mathcal{S}_n$ and if $k=2,$ $\mathcal{B}_{2n}^2$ is the hyperoctahedral group $\mathcal{H}_n.$ Thus, our outcome in \cite{Tout2018} can be seen as a generalization of the result of Farahat and Higman in \cite{FaharatHigman1959} and our result in \cite{Tout2017} giving a polynomiality property for the structure coefficients of the center of the hyperoctahedral group algebra.\\ 

In \cite{wang2004farahat}, Wang studied the centers of group algebras of wreath products $G \wr \mathcal{S}_n,$ for any finite group $G.$ He proved using the Farahat-Higman approach that the structure constants are polynomials in $n.$ The goal of this paper is to give a proof using Ivanov-Kerov approach to the polynomiality property of the structure coefficients of $Z(\mathbb{C}[\mathcal{B}_{kn}^k]).$ For this reason, we generalize the concept of partial permutations introduced by Ivanov and Kerov in \cite{Ivanov1999}. We define the term of $k$-partial permutations. These are block permutations defined on appropriate sets. When $k=1,$ a $1$-partial permutation will be a partial permutation in the sense defined by Ivanov and Kerov. Using $k$-partial permutations, we build a combinatorial algebra $\mathcal{I}_\infty^k$ which projects onto the center of the group $\mathcal{B}_{kn}^k$ algebra for each $n.$ We give two filtrations on $\mathcal{I}_\infty^k$ and we show that it is isomorphic to the algebra of shifted symmetric functions on $\wp(k)$ alphabets, where $\wp(k)$ is the number of partitions of the integer $k.$\\

The paper is organized as follows. In Section \ref{sec_2}, we give necessary definitions of partitions and we review the results concerning the conjugacy classes of the group $\mathcal{B}_{kn}^k$ presented in \cite{Tout2018}. Then, we introduce the notion of $k$-partial permutation in Section \ref{sec:k-partial}. We study some actions of the group $\mathcal{B}_{kn}^k$ on the set of $k$-partial permutations of $n$ and we build our main tool the combinatorial algebra $\mathcal{I}_\infty^k$ which projects on $Z(\mathbb{C}[\mathcal{B}_{kn}^k])$ for each $n.$ Next in Section \ref{sec:str_coef}, we prove our main result, Theorem \ref{Main_Theorem}, that gives a polynomiality property in $n$ for the structure coefficients of the algebra $Z(\mathbb{C}[\mathcal{B}_{kn}^k]).$ In addition, we give two filtrations on the algebra $\mathcal{I}_\infty^k$ and we provide explicit expressions for the structure coefficients in special cases of $k=1,$ $k=2$ and $k=3.$ In the last section, we show that the algebra $\mathcal{I}_\infty^k$ is isomorphic to the algebra $\mathcal{A}^{k*}$ of shifted symmetric functions on $\wp(k)$ alphabets.

\section{Conjugacy classes of the wreath product of symmetric groups}\label{sec_2}

In \cite{Tout2018} we introduced the group $\mathcal{B}_{kn}^{k}$ and we showed that it is isomorphic to the wreath product $\mathcal{S}_k\wr \mathcal{S}_n$ of the symmetric group $\mathcal{S}_k$ by the symmetric group $\mathcal{S}_n.$ In this section we will review all necessary definitions and results concerning this group.

\subsection{Partitions} A \textit{partition} $\lambda$ is a weakly decreasing list of positive integers $(\lambda_1,\ldots,\lambda_l)$ where $\lambda_1\geq \lambda_2\geq\ldots \geq\lambda_l\geq 1.$ The $\lambda_i$ are called the \textit{parts} of $\lambda$; the \textit{size} of $\lambda$, denoted by $|\lambda|$, is the sum of all of its parts. If $|\lambda|=n$, we say that $\lambda$ is a partition of $n$ and we write $\lambda\vdash n$. The number of parts of $\lambda$ is denoted by $l(\lambda)$. We will also use the exponential notation $\lambda=(1^{m_1(\lambda)},2^{m_2(\lambda)},3^{m_3(\lambda)},\ldots),$ where $m_i(\lambda)$ is the number of parts equal to $i$ in the partition $\lambda.$ In case there is no confusion, we will omit $\lambda$ from $m_i(\lambda)$ to simplify our notation. If $\lambda=(1^{m_1(\lambda)},2^{m_2(\lambda)},3^{m_3(\lambda)},\ldots,n^{m_n(\lambda)})$ is a partition of $n$ then $\sum_{i=1}^n im_i(\lambda)=n.$ We will dismiss $i^{m_i(\lambda)}$ from $\lambda$ when $m_i(\lambda)=0,$ for example, we will write $\lambda=(1^2,3,6^2)$ instead of $\lambda=(1^2,2^0,3,4^0,5^0,6^2,7^0).$ If $\lambda$ and $\delta$ are two partitions, we define the \textit{union} $\lambda \cup \delta$ and subtraction $\lambda \setminus \delta$ (if exists) as the following partitions:
$$\lambda \cup \delta=(1^{m_1(\lambda)+m_1(\delta)},2^{m_2(\lambda)+m_2(\delta)},3^{m_3(\lambda)+m_3(\delta)},\ldots).$$
$$\lambda \setminus \delta=(1^{m_1(\lambda)-m_1(\delta)},2^{m_2(\lambda)-m_2(\delta)},3^{m_3(\lambda)-m_3(\delta)},\ldots) \text{ if $m_i(\lambda)\geq m_i(\delta)$ for any $i.$ }$$
A partition is called \textit{proper} if it does not have any part equal to 1. The proper partition associated to a partition $\lambda$ is the partition $\bar{\lambda}:=\lambda \setminus (1^{m_1(\lambda)})=(2^{m_2(\lambda)},3^{m_3(\lambda)},\ldots).$ 

If $\lambda$ is a partition of $r<n,$ we can extend $\lambda$ to a partition of $n$ by adding $n-r$ parts equal to one, the new partition of $n$ will be denoted $\underline{\lambda}_n:$
$$\underline{\lambda}_n:=\lambda\cup (1^{n-|\lambda|}).$$

\subsection{Conjugacy classes of the symmetric group $\mathcal{S}_n$} The \textit{cycle-type} of a permutation of $\mathcal{S}_n$ is the partition of $n$ obtained from the lengths of the cycles that appear in its decomposition into a product of disjoint cycles. For example, the permutation $(2,4,1,6)(3,8,10,12)(5)(7,9,11)$ of $\mathcal{S}_{12}$ has cycle-type $(1,3,4^2).$ In this paper we will denote the cycle-type of a permutation $\omega$ by $\ct(\omega).$ It is well known that two permutations of $\mathcal{S}_n$ belong to the same conjugacy class if and only if they have the same cycle-type. Thus the conjugacy classes of the symmetric group $\mathcal{S}_n$ can be indexed by partitions of $n.$ If $\lambda=(1^{m_1(\lambda)},2^{m_2(\lambda)},3^{m_3(\lambda)},\ldots,n^{m_n(\lambda)})$ is a partition of $n,$ we will denote by $C_\lambda$ the conjugacy class of $\mathcal{S}_n$ associated to $\lambda:$
$$C_\lambda:=\lbrace \sigma\in \mathcal{S}_n \text{ $\mid$ } \ct(\sigma)=\lambda \rbrace.$$
The cardinality of $C_\lambda$ is given by:
$$|C_\lambda|=\frac{n!}{z_\lambda},$$
where
$$z_\lambda:=\prod_{r=1}^{n} r^{m_r(\lambda)}m_r(\lambda)!$$ 


\subsection{Conjugacy classes of the group $\mathcal{B}_{kn}^k$} We recall the definition of the group $\mathcal{B}_{kn}^k$ as given in \cite{Tout2018}. If $i$ and $k$ are two positive integers, we denote by $p_{k}(i)$ the following set of size $k:$ 
$$p_k(i):=\lbrace (i-1)k+1, (i-1)k+2, \cdots , ik\rbrace.$$ 

The set $p_k(i)$ will be called a $k$-tuple. The group $\mathcal{B}_{kn}^{k}$ is the subgroup of $\mathcal{S}_{kn}$ formed by permutations that send each set of the form $p_{k}(i)$ to another with the same form:
$$\mathcal{B}_{kn}^{k}:=\lbrace w \in \mathcal{S}_{kn}; \ \forall \ 1 \leq r \leq n, \ \exists \ 1 \leq r' \leq n \text{ such that } w(p_{k}(r))=p_{k}(r')\rbrace.$$
The order of the group $\mathcal{B}_{kn}^{k}$ is equal to
$$|\mathcal{B}_{kn}^{k}|=(k!)^nn!$$

In particular $\mathcal{B}_{n}^1$ is the symmetric group $\mathcal{S}_n$ and $\mathcal{B}_{2n}^2$ is the hyperoctahedral group $\mathcal{H}_n$ on $2n$ elements. In fact, as shown in \cite{Tout2018}, $\mathcal{B}_{kn}^{k}$ is isomorphic to the wreath product $\mathcal{S}_k\wr \mathcal{S}_n.$

For a permutation $\omega\in \mathcal{B}_{kn}^{k}$ and a partition $\rho=(\rho_1,\ldots,\rho_l)$ of $k$ we will construct the partition $\omega(\rho)$ as follows. First decompose $\omega$ as a product of disjoint cycles. Consider the collection of cycles $C_1,\ldots, C_l$ such that $C_1$ contains $\rho_1$ elements of a certain $k$-tuple $p_k(i),$ $C_2$ contains $\rho_2$ elements of the same $k$-tuple $p_k(i),$ etc. Now add the part $m$ to $\omega(\rho)$ if $m$ is the number of $k$-tuples that form the cycles $C_1,\ldots, C_l.$


\begin{ex}
Consider the following permutation, written in two-lines notation, $\omega$ of $\mathcal{B}^3_{15}$ 

$$\begin{pmatrix}
1&2&3&\mid&4&5&6&\mid&7&8&9&\mid&10&11&12&\mid&13&14&15&\mid&16&17&18\\
12&10&11&\mid&13&14&15&\mid&7&8&9&\mid&1&2&3&\mid&5&4&6&\mid&17&18&16
\end{pmatrix}.
$$
We put the sign $\mid$ after each three elements to emphasize that we are working in the case $k=3.$ The decomposition of $\omega$ into product of disjoint cycles is:
$$\omega=(1,12,3,11,2,10)(4,13,5,14)(6,15)(7)(8)(9)(16,17,18).$$
The first cycle $(1,12,3,11,2,10)$ contains all the elements of $p_3(1)$ thus it contributes to $\omega(3).$ In it, there are two $3$-tuples namely $p_3(1)$ and $p_3(4).$ Thus we should add a part $2$ to the partition $\omega(3).$ In the same way the cycle $(16,17,18)$ will add a part $1$ to $\omega(3)$ to become the partition $(2,1).$ By looking to the cycles $(4,13,5,14)(6,15)$ we see that $4$ and $5$ belong to the same cycle while $6$ belongs to the other, thus these cycles will contribute to $\omega(2,1).$ Since these cycles are formed by the two $3$-tuples $p_3(2)$ and $p_3(5),$ we have $\omega(2,1)=(2).$ The remaining cycles $(7)(8)(9)$ give $\omega(1,1,1)=(1).$
\end{ex}

\begin{definition} By a \emph{ family of partitions } we will always mean a family of partitions $\Lambda=(\Lambda(\lambda))_\lambda$ indexed by the partitions $\lambda$ of $k$ that satisfies:
$$|\Lambda|:=\sum_{\lambda\vdash k}|\Lambda(\lambda)|=n.$$
\end{definition}
 
\begin{definition}
If $\omega\in \mathcal{B}_{kn}^k,$ define $\ty(\omega)$ to be the following family of partitions
$$\ty(\omega):=(\omega(\rho))_{\rho\vdash k}.$$
\end{definition}

In \cite{Tout2018}, we showed the following two important results:
\begin{enumerate}
\item[1-] Let $\omega\in \mathcal{B}_{kn}^k$ and $p_\omega$ be the permutation of $n$ defined by $p_\omega(i)=j$ whenever $\omega(p_k(i))=p_k(j).$ Then we have:
$$\bigcup_{\rho\vdash k}\omega(\rho)=\ct(p_\omega) \text{ and } \sum_{\rho\vdash k}|\omega(\rho)|=n.$$
\item[2-] The conjugacy classes of the group $\mathcal{B}_{kn}^k$ are indexed by families of partitions and the associated conjugacy class for a given family of partitions $\Lambda$ is:
$$C_\Lambda:=\lbrace \omega\in \mathcal{B}_{kn}^k \text{ such that } \ty(\omega)=\Lambda\rbrace.$$ 
In addition:
$$|C_\Lambda|=\frac{n!(k!)^n}{Z_\Lambda}$$
where 
$$Z_\Lambda:=\displaystyle \prod_{\lambda\vdash k}z_{\Lambda(\lambda)}z_\lambda^{l(\Lambda(\lambda))}.$$
\end{enumerate}

\section{The algebra of $k$-partial permutations}\label{sec:k-partial}

In \cite{Ivanov1999}, Ivanov and Kerov introduced a useful tool called partial permutation to give a combinatorial proof to the polynomiality property of the center of the symmetric group algbera obtained by Farahat and Higman in \cite{FaharatHigman1959}. They showed that the algebra of partial permutations that are invariant under some action of the symmetric group is isomorphic to the algebra of shifted symmetric functions. In \cite{Tout2018}, we used the general framework built in \cite{Tout2017} to generalize the result of Farahat and Higman to wreath product of symmetric groups. We proved the polynomiality property for the structure coefficients of the center of the group $\mathcal{B}_{kn}^k$ algebra. The goal of this section is to generalize the concept of partial permutations in order to obtain a combinatorial proof for this result.

\subsection{$k$-partial permutations} The definition of the group $\mathcal{B}_{kn}^k$ can be extended to any set formed by a disjoint union of $k$-tuples as follows. Suppose we have a set $d$ that is a disjoint union of some $k$-tuples
$$d=\bigsqcup_{i=1}^rp_k(a_i),$$
where $a_i$ is a positive integer for any $1\leq i\leq r.$ We define the group $\mathcal{B}_d^k$ to be the following group of permutations:
$$\mathcal{B}_d^k:=\lbrace \omega\in \mathcal{S}_d \mid \forall 1\leq i \leq r, \exists 1\leq j\leq r \text{ with }\omega(p_k(a_i))=p_k(a_j) \rbrace,$$
where $\mathcal{S}_d$ is the group of permutations of the set $d.$ In other words, the group $\mathcal{B}_d^k$ consists of permutations that permute the blocks of the set $d.$

\begin{definition} Let $n$ be a non-negative integer. A $k$-partial permutation of $n$ is a pair $(d,\omega)$ where $d\subset [kn]$ is a disjoint union of some $k$-tuples and $\omega\in \mathcal{B}_d^k.$
\end{definition}

The concept of $k$-partial permutation can be seen as a generalization of the concept of a partial permutation defined by Ivanov and Kerov in \cite{Ivanov1999}. In fact when $k=1,$ a $1$-partial permutation is a partial permutation as defined in \cite{Ivanov1999}. We will denote by $\mathcal{P}_{kn}^k$ the set of all $k$-partial permutations of $n.$ It is clear that the cardinality of the set $\mathcal{P}_{kn}^k$ is 
$$|\mathcal{P}_{kn}^k|=\sum_{r=0}^n {n \choose r}(k!)^rr!=\sum_{r=0}^n (n\downharpoonright r)(k!)^r,$$
where $(n\downharpoonright r):=n(n-1)\cdots (n-r+1)$ is the falling factorial.
\begin{definition}\label{def_supp_ext}
If $(d,\omega)$ is a $k$-partial permutation of $n,$ we define:
\begin{enumerate}
\item[1.] $\supp(\omega)$ to be the support of $\omega.$ That is the minimal union of $k$-tuples of $d$ on which $\omega$ does not act like the identity.
\item[2.] $\underline{\omega}_n$ to be the permutation of the set $[kn]$ obtained from $\omega$ by natural extension (extension by identity).
\end{enumerate}
\end{definition}

\begin{ex} Consider the $3$-partial permutation $(d,\omega),$ where $d=p_3(1)\cup p_3(2)\cup p_3(4)\cup p_3(6)$ and
$$\omega=
\begin{pmatrix}
1&2&3&\mid&4&5&6&\mid&10&11&12&\mid&16&17&18\\
12&10&11&\mid&4&5&6&\mid&16&18&17&\mid&1&2&3
\end{pmatrix}.
$$
We have $\supp(\omega)=p_3(1)\cup p_3(4)\cup p_3(6)$ and 
$$\underline{\omega}_6=
\begin{pmatrix}
1&2&3&\mid&4&5&6&\mid&7&8&9&\mid&10&11&12&\mid&13&14&15&\mid&16&17&18\\
12&10&11&\mid&4&5&6&\mid&7&8&9&\mid&16&18&17&\mid&13&14&15&\mid&1&2&3
\end{pmatrix}.
$$
\end{ex}

The notion of $\ty$ defined for the permutations of $\mathcal{B}_{kn}^k$ can be extended to the $k$-partial permutations of $n.$ If $(d,\omega)$ is a $k$-partial permutation of $n,$ we define its type $\lambda=(\lambda(\rho))_{\rho\vdash k}$ to be the type of its permutation $\omega.$ For example the cycle decomposition of the $3$-partial permutation given in the above example is
$$(1,12,17,2,10,16)(3,11,18)(4)(5)(6)$$
and its type is formed by $\omega(2,1)=(3)$ and $\omega(1^3)=(1).$

\subsection{Action of $\mathcal{B}_{kn}^k$ on the set $\mathcal{P}_{kn}^k$} There is a natural product of $k$-partial permutations of $n$ given in the following definition.

\begin{definition}
If $(d_1,\omega_1)$ and $(d_2,\omega_2)$ are two $k$-partial permutations of $n,$ we define their product as follows:
\begin{equation*}
(d_1,\omega_1)(d_2,\omega_2)=(d_1\cup d_2,\omega_1\omega_2),
\end{equation*}
where the composition $\omega_1\omega_2$ is made after extending both $\omega_1$ and $\omega_2$ by identity to $d_1\cup d_2.$
\end{definition}

It is clear that the set $\mathcal{P}_{kn}^k$ equipped with the above product of $k$-partial permutations is a semi-group. That is the product is associative with identity element the $k$-partial permutation $(\emptyset, 1_\emptyset)$ where $1_\emptyset$ is the trivial permutation of the empty set. The group $\mathcal{B}_{kn}^k$ acts on the semi-group $\mathcal{P}_{kn}^k$ by the following action:
\begin{equation*}
\sigma . (d,\omega)=(\sigma(d),\sigma \omega \sigma^{-1}),
\end{equation*}
for any $\sigma\in \mathcal{B}_{kn}^k$ and $(d,\omega)\in \mathcal{P}_{kn}^k.$ We will use the term \emph{conjugacy class} to denote an orbit of this action and we will say that two elements of $\mathcal{P}_{kn}^k$ are conjugate if they belong to the same orbit. Two $k$-partial permutations $(d_1,\omega_1)$ and $(d_2,\omega_2)$ of $n$  are in the same conjugacy class if and only if there exists a permutation $\sigma\in \mathcal{B}_{kn}^k$ such that $(d_2,\omega_2)=(\sigma(d_1),\sigma \omega_1 \sigma^{-1}).$ That is $|d_1|=|d_2|$ and $\ty(\omega_1)=\ty(\omega_2).$ Thus, we have the following proposition.

\begin{prop}
The conjugacy classes of the action of the group $\mathcal{B}_{kn}^k$ on the set $\mathcal{P}_{kn}^k$ of $k$-partial permutations of $n$ can be indexed by families $\Lambda=(\Lambda(\lambda))_{\lambda\vdash k}$ with $|\Lambda|\leq n$ and for such a family, its associated conjugacy class is:
$$C_{\Lambda;n}:=\lbrace (d,\omega)\in \mathcal{P}_{kn}^k \text{ such that } |d|=k|\Lambda| \text{ and } \ty(\omega)=\Lambda\rbrace.$$
\end{prop}
We recall that a partition is called proper if it does not have any part equal to $1.$
\begin{definition} A family of partitions $\Lambda$ is called proper if the partition $\Lambda(1^k)$ is proper.
\end{definition}

If $\Lambda$ is a proper family of partitions with $|\Lambda|\leq n,$ we define $\underline{\Lambda}_n$ to be the family of partitions $\Lambda$ except that $\Lambda(1^k)$ is replaced by $\Lambda(1^k)\cup (1^{n-|\Lambda|}).$ It is clear that $|\underline{\Lambda}_n|=n.$

Consider now the following surjective homomorphism $\psi$ that extends $k$-partial permutations of $n$ to elements of $\mathcal{B}_{kn}^k:$

$$\begin{array}{ccccc}
\psi & : &\mathcal{P}^k_{kn} & \to & \mathcal{B}_{kn}^k \\
& & (d,\omega) & \mapsto &  \underline{\omega}_n,\\
\end{array}$$
where $\underline{\omega}_n$ is defined in Definition \ref{def_supp_ext}. Let $\Lambda$ be a family of partitions with $|\Lambda|\leq n$ and fix a permutation $x\in C_{\underline{\Lambda}_n},$ where $C_{\underline{\Lambda}_n}$ is the conjugacy class in $\mathcal{B}_{kn}^k$ associated to the family of partitions $\underline{\Lambda}_n.$ The inverse image of $x$ by $\psi$ is formed by all the $k$-partial permutations $(d,\omega)$ of $n$ that satisfy the following two conditions: 
$$d\supset \supp(x) \text{ and  $\omega$ coincide with $x$ on $d.$} $$

Since $\supp(\omega)$ consists of $|\Lambda|-m_1(\Lambda(1^k))$ $k$-tuples, there are
$${n-|\Lambda|+m_1(\Lambda(1^k))\choose m_1(\Lambda(1^k))}$$ 
elements in $\psi^{-1}(x).$ All of these elements are in $C_{\Lambda;n}$ and we recover all the elements of $C_{\Lambda;n}$ when $x$ runs through all the elements of $C_{\underline{\Lambda}_n}.$ Thus we get the following proposition.

\begin{prop}\label{prop_rel_conj}
If $\Lambda$ is a family of partitions with $|\Lambda|<n$ then:
$$|C_{\Lambda;n}|={n-|\Lambda|+m_1(\Lambda(1^k))\choose m_1(\Lambda(1^k))}|C_{\underline{\Lambda}_n}|.$$
\end{prop}

The action of the group $\mathcal{B}_{kn}^k$ on the set $\mathcal{P}_{kn}^k$ of $k$-partial permutations of $n$ can be extended linearly to an action of $\mathcal{B}_{kn}^k$ on the algebra $\mathbb{C}[\mathcal{P}_{kn}^k].$ The homomorphism $\psi$ can also be extended by linearity to become a surjective homomorphism between the algebras $\mathbb{C}[\mathcal{B}_{kn}^k]$ and $\mathbb{C}[\mathcal{P}_{kn}^k].$ For any $\sigma\in \mathcal{B}_{kn}^k$ and $(d,\omega)\in \mathcal{P}_{kn}^k$ we have:

$$\psi(\sigma.(d,\omega))=\psi(\sigma(d),\sigma\omega\sigma^{-1})=\underline{\sigma\omega\sigma^{-1}}_n=\sigma\underline{\omega}_n\sigma^{-1}=\sigma.\underline{\omega}_n=\sigma.\psi(d,\omega).$$

Let $\mathcal{I}_{kn}^k$ be the sub-algebra of $\mathbb{C}[\mathcal{P}_{kn}^k]$ generated by formal sums of the conjugacy classes $C_{\Lambda;n},$ then we have $\psi(\mathcal{I}_{kn}^k)=Z(\mathbb{C}[\mathcal{B}_{kn}^k]),$ where $Z(\mathbb{C}[\mathcal{B}_{kn}^k])$ is the center of the group algebra $\mathbb{C}[\mathcal{B}_{kn}^k]$ and by Proposition \ref{prop_rel_conj},
$$\psi(\mathbf{C}_{\Lambda;n})={n-|\Lambda|+m_1(\Lambda(1^k))\choose m_1(\Lambda(1^k))}\mathbf{C}_{\underline{\Lambda}_n},$$
for any family of partitions $\Lambda$ with $|\Lambda|<n.$

\subsection{The algebra $\mathcal{I}_\infty^k$}
Let $\mathbb{C}[\mathcal{P}_\infty^k]$ denote the algebra generated by all the $k$-partial permutations with a finite support. Any element $a\in \mathbb{C}[\mathcal{P}_\infty^k]$ can be canonically written as follows:
\begin{equation}\label{can_form}
a=\sum_{r=0}^\infty \sum_{d}\sum_{\omega \in \mathcal{B}_d^k}a_{d,\omega}(d,\omega),
\end{equation}
where the second sum runs through all the set $d$ that are unions of $r$ $k$-tuples and $a_{d,\omega}\in \mathbb{C}$ for any $(d,\omega).$ Denote by $\Proj_n$ the natural projection homomorphism of $\mathbb{C}[\mathcal{P}_\infty^k]$ on $\mathbb{C}[\mathcal{P}_{kn}^k],$ that is if $a\in \mathbb{C}[\mathcal{P}_\infty^k]$ is canonically written as in (\ref{can_form}), then 
$$\Proj_n(a)=\sum_{r=0}^n \sum_{d}\sum_{\omega \in \mathcal{B}_d^k}a_{d,\omega}(d,\omega),$$
where the second sum is now taken over all the sets $d\subset [kn]$ that are unions of $r$ $k$-tuples.  The pair $(\mathbb{C}[\mathcal{P}_\infty^k], \Proj_n)$ is the projective limit of the family $(\mathbb{C}[\mathcal{P}_{kn}^k])_{n\geq 1}$ equipped with the morphisms $\Proj_{nm}:\mathbb{C}[\mathcal{P}_{kn}^k]\rightarrow \mathbb{C}[\mathcal{P}_{km}^k]$ defined on the basis elements of $\mathbb{C}[\mathcal{P}_{kn}^k]$ by 
$$\Proj_{nm}(d,\omega):=\left\{
\begin{array}{ll}
      (d,\omega) & \text{ if } d\subseteq [km] \\
      0 & \text{ otherwise } \\
\end{array} 
\right. 
$$
whenever $m\leq n.$\\

Let $\mathcal{B}_\infty^k$ denote the infinite group of permutations permuting $k$-tuples. That means that any $x\in \mathcal{B}_\infty^k$ is a permutation that permutes only finitely many $k$-tuples, i.e. it has a finite support. The action of $\mathcal{B}_{kn}^k$ on $\mathbb{C}[\mathcal{P}_{kn}^k]$ can be generalized to an action of $\mathcal{B}_\infty^k$ on the algebra $\mathbb{C}[\mathcal{P}_\infty^k].$ In concordance with our notations, let us denote $\mathcal{I}_\infty^k$ the sub-algebra of all finite linear combinations of the conjugacy classes of this action. In other words, $\mathcal{I}_\infty^k$ is generated by the elements $\mathbf{C}_\Lambda,$ indexed by families of partitions, and defined by

$$\mathbf{C}_\Lambda=\sum_{(d,\omega)}(d,\omega),$$
where the sum runs over all $k$-partial permutations $(d,\omega)\in \mathcal{P}_{\infty}^k$ such that $d$ is a union of $|\Lambda|$ $k$-tuples and $\omega$ has type $\Lambda.$ It would be clear that $\Proj_n(\mathbf{C}_\lambda)=0$ if $|\Lambda|>n$ and if $|\Lambda|\leq n,$

$$\Proj_n(\mathbf{C}_\Lambda)=\mathbf{C}_{\Lambda;n}.$$

\section{Structure coefficients of the center of $\mathcal{B}_{kn}^{k}$ algebra}\label{sec:str_coef}

In this section we present our main result in Theorem \ref{Main_Theorem}, a polynomiality property in $n$ for the structure coefficients of the center of the group $\mathcal{B}_{kn}^{k}$ algebra. To show it we describe the structure coefficients of the algebra $\mathcal{I}_\infty^k$ then we apply the composition of the morphisms $\psi\circ \Proj_n$ defined in the above sections.

Let $\Lambda$ and $\Delta$ be two proper families of partitions with $|\Lambda|,|\Delta|\leq n.$ In the algebra $\mathcal{I}_\infty^k,$ we can write the product $\mathbf{C}_\Lambda\mathbf{C}_\Delta$ as a linear combination of the basis elements, that is 
\begin{equation}\label{str_coe_I_infty}
\mathbf{C}_\Lambda\mathbf{C}_\Delta=\sum_{\Gamma}c_{\Lambda\Delta}^\Gamma\mathbf{C}_\Gamma,
\end{equation}
where $\Gamma$ runs through some families of partitions and $c_{\Lambda\Delta}^\Gamma$ are non-negative integers independent of $n.$ If we apply $\Proj_n$ to this equality we get the following identity in $\mathcal{I}_{kn}^k:$
$$\mathbf{C}_{\Lambda;n}\mathbf{C}_{\Delta;n}=\sum_{\Gamma}c_{\Lambda\Delta}^\Gamma\mathbf{C}_{\Gamma;n}.$$

Since $\Lambda$ and $\Delta$ are proper, by applying $\psi$ to this equality we obtain using Proposition \ref{prop_rel_conj} the following identity in the center of the group $\mathcal{B}_{kn}^k$ algebra:

$$\mathbf{C}_{\underline{\Lambda}_n}\mathbf{C}_{\underline{\Delta}_n}=\sum_{\Gamma}c_{\Lambda\Delta}^\Gamma{n-|\Gamma|+m_1(\Gamma(1^k))\choose m_1(\Gamma(1^k))}\mathbf{C}_{\underline{\Gamma}_n}.$$

The sum over all the families of partitions in the above equation can be turned into a sum over all the proper families of partitions if we sum up all the partitions that give $\mathbf{C}_{\underline{\Gamma}_n}.$ Explicitly, we have

$$\mathbf{C}_{\underline{\Lambda}_n}\mathbf{C}_{\underline{\Delta}_n}=\sum_{\Gamma}\Big(\sum_{r=1}^{n-|\Gamma|}c_{\Lambda\Delta}^{\underline{\Gamma}_{|\Gamma|+r}}{n-|\Gamma|\choose r}\Big)\mathbf{C}_{\underline{\Gamma}_n},$$
where the sum now runs over all proper families of partitions. The sums over $\Gamma$ in the above equations are finite. That means there is a finite number of family partitions $\Gamma$ appearing in each equation. To see this, one needs to understand what is the form of the families of partitions $\Gamma$ that may appear in Equation (\ref{str_coe_I_infty}).   

For fixed three families of (not necessarily proper) partitions $\Lambda,$ $\Delta$ and $\Gamma,$ the coefficient $c_{\Lambda\Delta}^{\Gamma}$ in Equation (\ref{str_coe_I_infty}) counts the number of pairs of $k$-partial permutations $\big( (d_1,\omega_1),(d_2,\omega_2)\big)\in C_\Lambda\times C_\Delta$ such that $$(d_1,\omega_1).(d_2,\omega_2)=(d,\omega)$$ where $(d,\omega)\in C_\Gamma$ is a fixed $k$-partial permutation from a specific conjugacy class. We should remark that when multiplying $(d_1,\omega_1)$ by $(d_2,\omega_2),$ the permutation $\omega_1\omega_2$ acts on at most $k|\Lambda|+k|\Delta|$ elements. This means that each family of partitions $\Gamma$ that appears in the sum of Equation (\ref{str_coe_I_infty}) must verify the following condition:
\begin{equation}\label{majoration}
\max(|\Lambda|,|\Delta|)\leq |\Gamma|\leq |\Lambda|+|\Delta|.
\end{equation} 
In other words, we have showed that the function $\deg:\mathcal{I}_\infty^k \rightarrow \mathbb{N}$ defined on the basis elements of $\mathcal{I}_\infty^k$ by $\deg(\mathbf{C}_\Lambda)=|\Lambda|$ is a filtration on $\mathcal{I}_\infty^k.$ Another filtration of the algebra $\mathcal{I}_\infty^k$ will be given in Proposition \ref{prop_filt}. We are able now to state the main theorem of this paper.

\begin{theoreme}\label{Main_Theorem}
Let $\Lambda, \Delta$ and $\Gamma$ be three proper families of partitions satisfying $\max(|\Lambda|,|\Delta|)\leq |\Gamma|\leq |\Lambda|+|\Delta|.$ For any integer $n\geq |\Gamma|$ we have 
$$c_{\underline{\Lambda}_n\underline{\Delta}_n}^{\underline{\Gamma}_n}=\sum_{r=1}^{n-|\Gamma|}c_{\Lambda\Delta}^{\underline{\Gamma}_{|\Gamma|+r}}{n-|\Gamma|\choose r},$$
where $c_{\Lambda\Delta}^{\underline{\Gamma}_{|\Gamma|+r}}$ are non-negative integers independent of $n.$
\end{theoreme}

By Equation (\ref{majoration}), the integers $r$ in the above theorem satisfy the following inequality which leads us to the next corollary 
$$r\leq |\Lambda|+|\Delta|-|\Gamma|.$$

\begin{cor}\label{main_cor}
Let $\Lambda, \Delta$ and $\Gamma$ be three proper families of partitions satisfying $\max(|\Lambda|,|\Delta|)\leq |\Gamma|\leq |\Lambda|+|\Delta|.$ The structure coefficient $c_{\underline{\Lambda}_n\underline{\Delta}_n}^{\underline{\Gamma}_n}$ is a polynomial in $n$ with non-negative integer coefficients and of degree at most $|\Lambda|+|\Delta|-|\Gamma|.$  
\end{cor}

The result in Corollary \ref{main_cor}, was first given in \cite{Tout2018} as an application of the general framework for the structure coefficients of the centers of finite groups algebra built in \cite{Tout2017}.\\

Filtrations allow us to have more information about the families of partitions $\Gamma$ that appear in the expression of the product $\mathbf{C}_\Lambda\mathbf{C}_\Delta$ in the algebra $\mathcal{I}_\infty^k.$ The first filtration on $\mathcal{I}_\infty^k$ was given by $\deg.$ In the next proposition we give another one.

\begin{prop}\label{prop_filt}
The function $\deg_1:\mathcal{I}_\infty^k \rightarrow \mathbb{N}$ defined on the basis elements of $\mathcal{I}_\infty^k$ by $$\deg_1(\mathbf{C}_\Lambda)=|\Lambda|+m_1(\Lambda(1^k))$$ is a filtration.
\end{prop}
\begin{proof}
Let $(d_1,\omega_1)$ and $(d_2,\omega_2)$ be two $k$-partial permutations and suppose that:
$$d_1=\bigcup_{i=1}^rp_k(a_i) \text{ and } d_2=\bigcup_{j=1}^s p_k(b_j)$$
for some integers $a_i$ and $b_j.$ Denote by $A,$ $B,$ $H,$ $K,$ $G,$ $L,$ $T$ and $F$ the following sets:
$$A=\lbrace a_i \text{ such that } 1\leq i\leq r\rbrace, ~~ B=\lbrace b_j \text{ such that } 1\leq j\leq s\rbrace,$$
$$H=\lbrace a_i\in A\setminus B \text{ such that $\omega_1$ acts as identity on $p_k(a_i)$}\rbrace,$$
$$K=\lbrace a_i\in A\cap B \text{ such that $\omega_1$ acts as identity on $p_k(a_i)$}\rbrace,$$ 
$$G=\lbrace b_j\in B\setminus A \text{ such that $\omega_2$ acts as identity on $p_k(b_j)$}\rbrace,$$
$$L=\lbrace b_j\in A\cap B \text{ such that $\omega_2$ acts as identity on $p_k(b_j)$}\rbrace,$$
$$T=\lbrace x\in A\cap B \text{ such that $\omega_1$ and $\omega_2$ act as identities on $p_k(x)$}\rbrace$$
and 
$$F=\lbrace y\in (A\cap B)\setminus T \text{ such that $\omega_1$ and $\omega_2$ act as inverses on $p_k(y)$}\rbrace.$$
It would be easy to see that $|T|+|F|\leq |A\cap B|+|K|+|L|$ and from it obtain the following: 
\begin{eqnarray*}
\deg_1(d_1\cup d_2,\omega_1\omega_2)&=&|A\setminus B|+|B\setminus A|+|A\cap B|+|H|+|G|+|T|+|F|\\
&\leq & \underbrace{|A\setminus B|+|A\cap B|+|H|+|K|}_{\deg_1(d_1,\omega_1)}+\underbrace{|B\setminus A|+|A\cap B|+|G|+|L|}_{\deg_1(d_2,\omega_2)}.
\end{eqnarray*}
The result follows.
\end{proof} 

In \cite{Ivanov1999}, the authors propose many filtrations that may be generalized to our case. We turn now to make some computations. In the following three examples, we give explicit expressions of products of conjugacy classes for the cases $k=1,$ $k=2$ and $k=3$ respectively. 

\begin{ex} \label{ex_k=1} Let $k=1.$ Since $(1)$ is the only partition of one, the algebra $\mathcal{I}_\infty^1$ is generated by partitions. For instance, $C_{(2)}$ is the set of all $1$-partial permutations that have type $(2)$ or equivalently the set of all partial permutations with cycle-type $(2).$ In $\mathcal{I}_\infty^1$ we have for example
$$\mathbf{C}_{(2)}\mathbf{C}_{(2)}=\mathbf{C}_{(1^2)}+3\mathbf{C}_{(3)}+2\mathbf{C}_{(2^2)}$$
and 
$$\mathbf{C}_{(2)}\mathbf{C}_{(3)}=2\mathbf{C}_{(1,2)}+4\mathbf{C}_{(4)}+\mathbf{C}_{(2,3)}.$$
The first equation appears in \cite{Ivanov1999} and the second appears in \cite{touPhd14}. Apply now $\psi\circ \Proj_n$ on the above two expressions to get the following results in the center of the symmetric group algebra $Z(\mathbf{C}[\mathcal{S}_n]]):$
$$\mathbf{C}_{(1^{n-2},2)}\mathbf{C}_{(1^{n-2},2)}=\frac{n(n-1)}{2} \mathbf{C}_{(1^n)}+3\mathbf{C}_{(1^{n-3},3)}+2\mathbf{C}_{(1^{n-4},2^2)} \text{ for any $n\geq 4$},$$
and 
$$\mathbf{C}_{(1^{n-2},2)}\mathbf{C}_{(1^{n-3},3)}=2(n-2) \mathbf{C}_{(1^{n-2},2)}+4\mathbf{C}_{(1^{n-4},4)}+\mathbf{C}_{(1^{n-5},2,3)} \text{ for any $n\geq 5$}.$$
\end{ex}

\begin{ex} \label{ex_k=2} There are only two partitions of $2,$ namely $(1^2)$ and $(2).$ Thus the elements generating $\mathcal{I}_\infty^2$ are indexed by families of partitions $\Lambda=(\Lambda(1^2),\Lambda(2)).$ Take for instance $\Lambda=((1),(2)),$ then $C_{((1),(2))}$ is the set of all $2$-partial permutations with type $((1),(2)).$ For example, $(\lbrace 3,4,7,8,9,10\rbrace, (3,7,4,8)(9)(10))\in C_{((1),(2))}.$ We have the following two complete expressions in $\mathcal{I}_\infty^2:$
\begin{equation}
\mathbf{C}_{((1),\emptyset)}\mathbf{C}_{((1),(1))}=2\mathbf{C}_{((1),(1))}+2\mathbf{C}_{((1^2),(1))}
\end{equation}
and 
\begin{equation}
\mathbf{C}_{(\emptyset,(2))}\mathbf{C}_{(\emptyset,(2))}=2\mathbf{C}_{((1^2),\emptyset)}+2\mathbf{C}_{(\emptyset,(1^2))}+2\mathbf{C}_{(\emptyset,(2^2))}+3\mathbf{C}_{((3),\emptyset)}.
\end{equation} 
For example the first coefficient $2$ in the above first equation is due to the fact that there are only two pairs $(x,y)\in C_{((1),\emptyset)}\times C_{((1),(1))}$ that satisfy $xy=(\lbrace 1,2,3,4\rbrace;(1)(2)(3,4)).$ Namely $(x,y)$ can be one and only one of the following pairs:

$$\Big( \big(\lbrace 1,2\rbrace;(1)(2)\big),\big(\lbrace 1,2,3,4\rbrace;(1)(2)(3,4)\big)\Big)\text{ or }\Big( \big(\lbrace 3,4\rbrace;(3)(4)\big),\big(\lbrace 1,2,3,4\rbrace;(1)(2)(3,4)\big)\Big).$$

Apply now $\psi\circ \Proj_n$ for the above two expressions to get the following results in the center of the hyperoctahedral group algebra $Z(\mathbf{C}[\mathcal{B}_{2n}^2]):$
$$\mathbf{C}_{((1^n),\emptyset)}\mathbf{C}_{((1^{n-1}),(1))}=\mathbf{C}_{((1^{n-1}),(1))} \text{ for any $n\geq 3$}$$
and for any $n\geq 5$
$$\mathbf{C}_{((1^{n-2}),(2))}\mathbf{C}_{((1^{n-2}),(2))}=n(n-1) \mathbf{C}_{((1^{n}),\emptyset)}+2\mathbf{C}_{((1^{n-2}),(1^2))}+2\mathbf{C}_{((1^{n-4}),(2^2))}+3\mathbf{C}_{((1^{n-3},3),\emptyset)}.$$
The first equation comes with no surprise since $C_{((1^n),\emptyset)}$ is the identity class.
\end{ex}

\begin{ex} $(1^3),$ $(2,1)$ and $(3)$ are the only three partitions of $3.$ Thus $\mathcal{I}_\infty^3$ is generated by elements indexed by triplets of partitions. We will suppose that the first corresponds to the partition $(1^3),$ the second corresponds to $(2,1)$ and the third corresponds to $(3).$ For instance $\mathbf{C}_{(\emptyset,(1),(1))}$ contains the $3$-partial permutation $\big( \lbrace  4,5,6,13,14,15\rbrace;(4,5)(6)(13,15,14)\big).$ In $\mathcal{I}_\infty^3$ we have the following two complete products:

$$\mathbf{C}_{(\emptyset,(1),(1))}\mathbf{C}_{(\emptyset,\emptyset,(1))}=2\mathbf{C}_{(\emptyset,(1),(1^2))}+2\mathbf{C}_{((1),(1),\emptyset)}+3\mathbf{C}_{(\emptyset,(1),(1))}$$
and
$$\mathbf{C}_{(\emptyset,(1),(1))}\mathbf{C}_{(\emptyset,(1),\emptyset)}=2\mathbf{C}_{(\emptyset,(1^2),(1))}+3\mathbf{C}_{((1),\emptyset,(1))}+4\mathbf{C}_{(\emptyset,(1^2),\emptyset)}+6\mathbf{C}_{(\emptyset,\emptyset,(1^2))}.$$
For example the coefficient $3$ in the above first equation comes from the fact that:
\begin{eqnarray*}
(1,2)(3)(7,8,9)&=&(1,2)(3)(7,9,8).(7,9,8)\\
&=&(1,3)(2)(7,8,9).(1,2,3)\\
&=&(2,3)(1)(7,8,9).(1,3,2).
\end{eqnarray*}

If we apply now $\psi\circ \Proj_n$ on the above equations, we get the following explicit expressions in the center of the group $\mathcal{B}_{3n}^3.$ For $n\geq 2:$
$$\mathbf{C}_{((1^{n-2}),(1),(1))}\mathbf{C}_{((1^{n-1}),\emptyset,(1))}=2\mathbf{C}_{((1^{n-3}),(1),(1^2))}+2(n-1)\mathbf{C}_{((1^{n-1}),(1),\emptyset)}+3\mathbf{C}_{((1^{n-2}),(1),(1))}$$
and for $n\geq 3:$
\begin{eqnarray*}
\mathbf{C}_{((1^{n-2}),(1),(1))}\mathbf{C}_{((1^{n-1}),(1),\emptyset)}&=&2\mathbf{C}_{((1^{n-3}),(1^2),(1))}+3(n-1)\mathbf{C}_{((1^{n-1}),\emptyset,(1))}\\
&&+4\mathbf{C}_{((1^{n-2}),(1^2),\emptyset)}+6\mathbf{C}_{((1^{n-2}),\emptyset,(1^2))}.
\end{eqnarray*}

\end{ex}

\section{Irreducible characters of $\mathcal{B}_{kn}^k$ and symmetric functions}\label{sec:iso}

In \cite{Ivanov1999}, Ivanov and Kerov showed that the algebra $\mathcal{I}_\infty^1$ is isomorphic to the algebra of shifted symmetric functions which we shall denote $\mathcal{A}^{1*}.$ The goal of this section is to prove that in general for any fixed positive integer $k,$ the algebra $\mathcal{I}_\infty^k$ is isomorphic to an algebra of shifted symmetric functions on many alphabets denoted by $\mathcal{A}^{k*}.$ We start first by recalling the result of Ivanov and Kerov in \cite{Ivanov1999}. Then we show a similar result in the case $k=2$ (hyperoctahedral group) before considering the general case.

\subsection{Case $k=1$: the symmetric group $\mathcal{S}_n$} The irreducible $\mathcal{S}_n$-modules are indexed by partitions of $n.$ For $\lambda\vdash n,$ we will denote by $V^\lambda$ its associated irreducible $\mathcal{S}_n$-module and by $\chi^\lambda$ its character. The algebra $\mathcal{A}^1$ of symmetric functions has many basis families indexed by partitions. Among them are the power sum functions $(p_\lambda)_{\lambda}$ and the Schur functions $(s_\lambda)_{\lambda}.$ The transition matrix between these two basis is given by the following formula of Frobenius:
\begin{equation} \label{Frob_form}
p_\delta=\sum_{\rho \atop{|\rho|=|\delta|}}\chi^{\rho}_{\delta}s_\rho,
\end{equation}
where $\chi^{\rho}_{\delta}$ denotes the value of the character $\chi^{\rho}$ on any permutation of cycle-type $\delta.$

A shifted symmetric function $f$ in infinitely many variables $(x_1,x_2,\ldots)$ is a family $(f_i)_{i>1}$ that satisfies the following two properties: 
\begin{enumerate}
\item[1.] $f_i$ is a symmetric polynomial in $(x_1 -1,x_2 -2 ,\ldots ,x_i-i).$ 
\item[2.] $f_{i+1}(x_1,x_2,\ldots ,x_i,0) = f_i(x_1,x_2,\ldots ,x_i).$ 
\end{enumerate} 
The set of all shifted symmetric functions is an algebra which we shall denote $\mathcal{A}^{1*}.$ It has many basis families indexed by partitions. In \cite{okounkov1997shifted}, Okounkov and Olshanski gave a linear isomorphism $\varphi:\mathcal{A}^{1}\rightarrow \mathcal{A}^{1*}.$ For any partition $\lambda,$ the images of the power sum function $p_\lambda$ and the Schur function $s_\lambda$ by $\varphi$ are the shifted power symmetric function $p^{\#}_\lambda$ and the shifted Schur function $s^*_\lambda.$ By applying $\varphi$ to the Frobenius relation given in Equation (\ref{Frob_form}), we get:

\begin{equation}\label{shif_Frob_form}
p^{\#}_\delta=\sum_{\rho \atop{|\rho|=|\delta|}}\chi^{\rho}_{\delta}s^*_\rho.
\end{equation}

If $f\in \mathcal{A}^{1*}$ and if $\lambda=(\lambda_1,\lambda_2,\cdots ,\lambda_l)$ is a partition, we denote by $f(\lambda)$ the value $f_l(\lambda_1,\lambda_2,\cdots ,\lambda_l).$ By \cite{okounkov1997shifted}, any shifted symmetric function is determined by its values on partitions. The vanishing characterization of the shifted symmetric functions given in \cite{okounkov1997shifted} states that $s_\rho^*$ is the unique shifted symmetric function of degree at most $|\rho|$ such that 
\begin{equation}\label{characterisation}
s^{*}_\rho(\lambda)=\left\{
\begin{array}{ll}
      \frac{(|\lambda|\downharpoonright |\rho|)}{\dim \lambda} f^{\lambda/\rho}& \text{ if } \rho\subseteq \lambda \\
      0 & \text{ otherwise } \\
\end{array} 
\right. 
\end{equation}
where $(|\lambda|\downharpoonright |\rho|):=|\lambda|(|\lambda|-1)\cdots (|\lambda|-|\rho|+1)$ is the falling factorial and $f^{\lambda/\rho}$ is the number of skew standard tableau of shape $\lambda/\rho.$ Using the following branching rule for characters of the symmetric groups
\begin{equation}\label{branching_rule}
\chi^{\lambda}_{\rho\cup (1^{|\lambda|-|\rho|})}=\sum_{\nu; |\nu|=|\rho|}f^{\lambda/\nu}\chi^\nu_\rho,
\end{equation}
one can verify using formulas (\ref{shif_Frob_form}) and (\ref{characterisation}) that
$$p^{\#}_\delta(\lambda)=\left\{
\begin{array}{ll}
      \frac{(|\lambda|\downharpoonright |\delta|)}{\dim \lambda} \chi^{\lambda}_{\underline{\delta}_{|\lambda|}}& \text{ if } |\lambda|\geq |\delta| \\
      0 & \text{ otherwise } \\
\end{array} 
\right. 
$$

If $k=1,$ the algebra $\mathcal{I}_\infty^1$ constructed in the previous section has a basis $\mathbf{C}_\lambda$ indexed by partitions. Now let $\lambda$ be any partition and consider the composition $F^1_\lambda:=\frac{\chi^\lambda}{\dim \lambda}\circ \psi \circ \Proj_{|\lambda|}$ of morphisms, where $\dim \lambda$ denotes the dimension of the irreducible $\mathcal{S}_n$-module $V^\lambda.$ If $\delta$ is a partition such that $|\delta|>|\lambda|,$ then $F^1_\lambda(\mathbf{C}_\delta)=0$ since the projection is zero in this case. Suppose now that $|\lambda|\geq |\delta|,$ we have the following equalities:
\begin{eqnarray}
\left(\frac{\chi^\lambda}{\dim \lambda}\circ \psi \circ \Proj_{|\lambda|} \right)(\mathbf{C}_{\delta})&=& \frac{\chi^\lambda}{\dim \lambda} \left( {|\lambda|-|\delta|+m_1(\delta)\choose m_1(\delta)} \mathbf{C}_{\underline{\delta}_{|\lambda|}} \right) \\
\notag
&=& {|\lambda|-|\delta|+m_1(\delta)\choose m_1(\delta)}\frac{|\lambda|!}{ z_{\underline{\delta}_{|\lambda|}}\dim \lambda}\chi^{\lambda}_{\underline{\delta}_{|\lambda|}} \\
\notag
&=& \frac{1}{z_\delta}\frac{(|\lambda|\downharpoonright |\delta|)}{\dim \lambda} \chi^{\lambda}_{\underline{\delta}_{|\lambda|}}.
\end{eqnarray}

This implies that $F^1_\lambda(\mathbf{C}_\delta)=z_\delta^{-1}p^{\#}_\delta(\lambda)$ and that the map $F^1:\mathcal{I}_\infty^1\rightarrow \mathcal{A}^{1*}$ defined on the basis elements of $\mathcal{I}_\infty^1$ by $$F^1(\mathbf{C}_\delta)=z_\delta^{-1}p^{\#}_\delta$$ is an isomorphism of algebras. This result was first shown by Ivanov and Kerov in \cite[Theorem 9.1]{Ivanov1999}.  It can be used to obtain the multiplication table of $p^{\#}_\delta$ in $\mathcal{A}^{1*}$ from that of $\mathbf{C}_\delta$ in $\mathcal{I}_\infty^1.$ For example, the two equations given in Example \ref{ex_k=1} implies the following two equalities in $\mathcal{A}^{1*}$
$$p^{\#}_{(2)}p^{\#}_{(2)}=p^{\#}_{(1^2)}+4p^{\#}_{(3)}+2p^{\#}_{(2^2)}$$
and 
$$p^{\#}_{(2)}p^{\#}_{(3)}=6p^{\#}_{(1,2)}+6p^{\#}_{(4)}+p^{\#}_{(2,3)}.$$
\subsection{Case $k=2$: the hyperoctahedral group $\mathcal{B}_{2n}^{2}$} 
The conjugacy classes as well as the irreducible representations of the hyperoctahedral group $\mathcal{B}_{2n}^2$ are indexed by bipartitions of $n.$ These are pairs of partitions $(\lambda_1,\lambda_2)$ satisfying $|\lambda_1|+|\lambda_2|=n,$ where $\lambda_1$ (resp. $\lambda_2$) corresponds to the partition $(1^2)$ (resp. $(2)$) of $2.$ To simplify our notations, $(\lambda_1,\lambda_2)\vdash n$ will be used to say that $(\lambda_1,\lambda_2)$ is a bipartition of $n.$ If $(\lambda_1,\lambda_2)\vdash n$ then the size of the conjugacy class and the dimension of the irreducible representation associated to $(\lambda_1,\lambda_2)$ are given respectively by the following formulas:
\begin{equation}
|C_{(\lambda_1,\lambda_2)}|=\frac{2^nn!}{z_{(\lambda_1,\lambda_2)}}=\frac{2^nn!}{2^{l(\lambda_1)}z_{\lambda_1}2^{l(\lambda_2)}z_{\lambda_2}}
\end{equation}
and 
\begin{equation}
\dim (\lambda_1,\lambda_2)=n! \frac{\dim \lambda_1}{(|\lambda_1|)!}\frac{\dim \lambda_2}{(|\lambda_2|)!}.
\end{equation}

If $\alpha=(\alpha_1,\alpha_2,\cdots, \alpha_r)$ is a partition of $n,$ define 
$$p_{\alpha}^{+}(x,y):=p_{\alpha_1}^{+}p_{\alpha_1}^{+}\cdots p_{\alpha_r}^{+} ~~~~ \text{ and } ~~~~ p_{\alpha}^{-}(x,y):=p_{\alpha_1}^{-}p_{\alpha_1}^{-}\cdots p_{\alpha_r}^{-}$$
where for any $k\geq 1$
$$p_k^+=\sum_{i\geq 1}(x_i^k+y_i^k) ~~~~ \text{ and } ~~~~ p_k^-=\sum_{i\geq 1}(x_i^k-y_i^k).$$

If $(\delta_1,\delta_2)$ is a bipartition we define the following two functions on two alphabets:
\begin{equation*}
p_{(\delta_1,\delta_2)}(x,y):=p_{\delta_1}^{+}(x,y)p_{\delta_2}^{-}(x,y) ~~~~ \text{ and } ~~~~ s_{(\rho_1,\rho_2)}(x,y):=s_{\rho_1}(x)s_{\rho_2}(y).
\end{equation*}
These are the generalizations of the power sum symmetric function and the Schur function to two alphabets and they form linear basis for the algebra $\mathcal{A}^{2}$ of symmetric functions on two alphabets. If $(\delta_1,\delta_2)\vdash n,$ the following formula in $\mathcal{A}^{2}$ is analogous to Formula (\ref{Frob_form}) of Frobenius, see \cite[page 178]{McDo} and \cite{adin2017character}:

\begin{equation*}
p_{(\delta_1,\delta_2)}(x,y)=\sum_{(\rho_1,\rho_2)\vdash n}  \chi^{(\rho_1,\rho_2)}_{(\delta_1,\delta_2)}s_{(\rho_1,\rho_2)}(x,y),
\end{equation*}
where $\chi^{(\rho_1,\rho_2)}_{(\delta_1,\delta_2)}$ denotes the value of the irreducible character $\chi^{(\rho_1,\rho_2)}$ of $\mathcal{B}_{2n}^2$ evaluated on a permutation with $2$-set type $(\delta_1,\delta_2).$ In the same way as we did in the previous section, we define the algebra $\mathcal{A}^{2*}$ of shifted functions in two alphabets and we have:
\begin{equation}\label{Frob_hyp}
p^{\#}_{(\delta_1,\delta_2)}(x,y)=\sum_{(\rho_1,\rho_2)\vdash n}  \chi^{(\rho_1,\rho_2)}_{(\delta_1,\delta_2)}s^*_{(\rho_1,\rho_2)}(x,y).
\end{equation}


\begin{ex}\label{ex_char_hyp} If $\delta_1=\delta_2=(1)$ we have the following equation:
$$\big(\sum_{i\geqslant 1}x_i+\sum_{i\geqslant 1}y_i\big)\big(\sum_{i\geqslant 1}x_i-\sum_{i\geqslant 1}y_i\big)=\chi_{((1),(1))}^{((2),\emptyset)}s_{(2)}({\bf{x}})+\chi_{((1),(1))}^{((1),(1))}s_{(1)}({\bf{x}})s_{(1)}({\bf{y}})$$
$$+\chi_{((1),(1))}^{(\emptyset,(2))}s_{(2)}({\bf{y}})+\chi_{((1),(1))}^{((1^2),\emptyset)}s_{(1^2)}({\bf{x}})+\chi_{((1),(1))}^{(\emptyset,(1^2))}s_{(1^2)}({\bf{y}}),$$
where $$s_{(1)}({\bf{x}})=\sum_{i\geqslant 1} x_i,$$
$$s_{(1^2)}({\bf{x}})=\sum_{1\leq x_i<x_j} x_ix_j$$ 
and $$s_{(2)}({\bf{x}})=\sum_{i\geqslant 1} x_i^2+\sum_{1\leqslant i_1<i_2}x_{i_1}x_{i_2}.$$
By identifying both sides we deduce that $\chi_{((1),(1))}^{((2),\emptyset)}=1,$ $\chi_{((1),(1))}^{((1),(1))}=0,$ $\chi_{((1),(1))}^{(\emptyset,(2))}=-1,$ $\chi_{((1),(1))}^{((1^2),\emptyset)}=1$ and $\chi_{((1),(1))}^{(\emptyset,(1^2))}=-1.$
\end{ex}
The characters $\chi^{(\rho_1,\rho_2)}_{(\delta_1,\delta_2)}$ can be expressed in terms of the characters of the symmetric group as following, see \cite[Formula (4.2)]{adin2017character}:

\begin{equation}\label{char_hyper}
\chi^{(\rho_1,\rho_2)}_{(\delta_1,\delta_2)}=\sum_{u,v}(-1)^{\mid \lbrace j: v_j=- \rbrace\mid} \chi^{\rho_1}_{\alpha_{uv}}\chi^{\rho_2}_{\beta_{uv}}
\end{equation}
where the sum ranges over all vectors 
$$u=(u_1,u_2,\cdots ,u_{l(\delta_1)})\in \lbrace -,+\rbrace^{l(\delta_1)} ~~~~ \text{ and } ~~~~ v=(v_1,v_2,\cdots ,v_{l(\delta_2)})\in \lbrace -,+\rbrace^{l(\delta_2)},$$
and where $\alpha_{uv}$ (resp. $\beta_{uv}$) is the composition consisting of the parts $(\delta_{1})_i$ of $\delta_1$ with $u_i=+$ (resp. $u_i=-$) followed by the parts $(\delta_{2})_j$ of $\delta_2$ with $v_j=+$ (resp. $v_j=-$). As application, we can recover the results of Example \ref{ex_char_hyp}. For instance:
$$\chi_{((1),(1))}^{((2),\emptyset)}=\sum_{(u,v)\in \lbrace(+,+)\rbrace}(-1)^{\mid (u,v): v=- \rbrace\mid} \chi^{(2)}_{\alpha_{uv}}\chi^{\emptyset}_{\beta_{uv}}=\chi^{(2)}_{(1^2)}=1,$$
and
$$\chi_{((1),(1))}^{((1),(1))}=\sum_{(u,v)\in \lbrace(-,+),(+,-)\rbrace}(-1)^{\mid (u,v): v=- \rbrace\mid} \chi^{(1)}_{\alpha_{uv}}\chi^{(1)}_{\beta_{uv}}=\chi^{(1)}_{(1)}\chi^{(1)}_{(1)}-\chi^{(1)}_{(1)}\chi^{(1)}_{(1)}=0.$$

\begin{prop} If $(\rho_1,\rho_2)\vdash n$ and $(\delta_1,\delta_2)\vdash r$ with $r\leq n$ then:
\begin{equation}\label{branching_hyp}
\chi^{(\rho_1,\rho_2)}_{\underline{(\delta_1,\delta_2)}_n}=\sum_{(\nu_1,\nu_2)\vdash r}\chi^{(\nu_1,\nu_2)}_{(\delta_1,\delta_2)}{n-r \choose |\rho_1|-|\nu_1|}f^{\rho_1/\nu_1}f^{\rho_2/\nu_2},
\end{equation}
and 
\begin{equation}
\frac{(n\downharpoonright r)}{\dim (\rho_1,\rho_2)}\chi^{(\rho_1,\rho_2)}_{\underline{(\delta_1,\delta_2)}_n}=p_{(\delta_1,\delta_2)}^{\#}(\rho_1,\rho_2).
\end{equation}
\end{prop}
\begin{proof} The first result is a consequence of Formulas (\ref{branching_rule}) and (\ref{char_hyper}). Multiplying Equation (\ref{branching_hyp}) by $\frac{(n\downharpoonright r)}{\dim (\rho_1,\rho_2)}$ yields:

\begin{eqnarray*}
\frac{(n\downharpoonright r)}{\dim (\rho_1,\rho_2)}\chi^{(\rho_1,\rho_2)}_{(\delta_1\cup (1^{n-r}),\delta_2)}&=&\sum_{(\nu_1,\nu_2)\vdash r}\frac{|\rho_1|!|\rho_2|!}{(n-r)!\dim \rho_1 \dim \rho_2}\chi^{(\nu_1,\nu_2)}_{(\delta_1,\delta_2)}{n-r \choose |\rho_1|-|\nu_1|}f^{\rho_1/\nu_1}f^{\rho_2/\nu_2}\\
&=&\sum_{(\nu_1,\nu_2)\vdash r} \chi^{(\nu_1,\nu_2)}_{(\delta_1,\delta_2)}\frac{(|\rho_1|\downharpoonright |\nu_1|)}{\dim \rho_1} f^{\rho_1/\nu_1} \frac{(|\rho_2|\downharpoonright |\nu_2|)}{\dim \rho_2}f^{\rho_2/\nu_2}\\
&=&\sum_{(\nu_1,\nu_2)\vdash r} \chi^{(\nu_1,\nu_2)}_{(\delta_1,\delta_2)}s^*_{\delta_1}(\rho_1)s^*_{\delta_2}(\rho_2)  ~~~~~~~~ (\text{ by characterization (\ref{characterisation}) })   \\
&=&p_{(\delta_1,\delta_2)}^{\#}(\rho_1,\rho_2)  ~~~~~~~~ (\text{ by Equation (\ref{Frob_hyp}) })
\end{eqnarray*}
\end{proof}

If $(\rho_1,\rho_2)\vdash n$ and $(\delta_1,\delta_2)\vdash r,$ applying the composition $F_{(\rho_1,\rho_2)}^2:=\frac{\chi^{(\rho_1,\rho_2)}}{\dim {(\rho_1,\rho_2)}}\circ \psi \circ \Proj_{n}$ of morphisms to the basis element of $\mathcal{I}_\infty^2$ indexed by the bipartition $(\delta_1,\delta_2)$ yields zero if $r>n$ and if $r\leq n$ we have: 
\begin{eqnarray}\label{eq_isom}
F_{(\rho_1,\rho_2)}^2(\mathbf{C}_{(\delta_1,\delta_2)})&=& \frac{\chi^{(\rho_1,\rho_2)}}{\dim {(\rho_1,\rho_2)}} \Big( {n-r+m_1(\delta_1)\choose m_1(\delta_1))} \mathbf{C}_{\underline{(\delta_1,\delta_2)}_{n}} \Big) \\
\notag
&=& {n-r+m_1(\delta_1)\choose m_1(\delta_1)}\frac{n!2^{n}}{ z_{\underline{(\delta_1,\delta_2)}_{n}}\dim (\rho_1,\rho_2)}\chi^{(\rho_1,\rho_2)}_{\underline{(\delta_1,\delta_2)}_{n}} \\
\notag
&=& {n-r+m_1(\delta_1)\choose m_1(\delta_1)}\frac{n!2^{n}}{ z_{(\delta_1,\delta_2)\frac{2^{n-r}(n-r+m_1(\delta_1))!}{m_1(\delta_1)!}}\dim (\rho_1,\rho_2)}\chi^{(\rho_1,\rho_2)}_{\underline{(\delta_1,\delta_2)}_{n}} \\
\notag
&=& \frac{2^{r}}{z_{(\delta_1,\delta_2)}}\frac{(n\downharpoonright r)}{\dim (\rho_1,\rho_2)} \chi^{(\rho_1,\rho_2)}_{\underline{(\delta_1,\delta_2)}_{n}}\\
\notag
&=&\frac{2^{r}}{z_{(\delta_1,\delta_2)}}p_{(\delta_1,\delta_2)}^{\#}(\rho_1,\rho_2).
\end{eqnarray}
This implies the following theorem.

\begin{theoreme}
The linear map $F^2:\mathcal{I}_\infty^2\longrightarrow \mathcal{A}^{2*}$ defined by $$F^2(\mathbf{C}_{(\delta_1,\delta_2)})=\frac{2^{|\delta_1|+|\delta_2|}}{z_{(\delta_1,\delta_2)}}p^{\#}_{(\delta_1,\delta_2)}$$ is an isomorphism of algebras.
\end{theoreme}

This isomorphism allows us to obtain the multiplication table of $p^{\#}_{(\delta_1,\delta_2)}$ in $\mathcal{A}^{2*}$ from the multiplication table of $\mathbf{C}_{(\delta_1,\delta_2)}$ in $\mathcal{I}_\infty^2.$ For instance, the equations given in Example \ref{ex_k=2} give us:

$$p^{\#}_{((1),\emptyset)}p^{\#}_{((1),(1))}=2p^{\#}_{((1),(1))}+p^{\#}_{((1^2),(1))}$$
and 
$$p^{\#}_{(\emptyset,(2))}p^{\#}_{(\emptyset,(2))}=p^{\#}_{((1^2),\emptyset)}+p^{\#}_{(\emptyset,(1^2))}+p^{\#}_{(\emptyset,(2^2))}+4p^{\#}_{((3),\emptyset)}.$$ 

\subsection{The general case.} We refer to \cite[Appendix B]{McDo} for the results of the representation theory of wreath products presented in this section. Let $(P_r(C_\lambda))_{r\geq 1, \lambda\vdash k}$ be a family of independent indeterminates over $\mathbb{C}.$ For each $\lambda\vdash k,$ we may think of $P_r(C_\lambda)$ as the $r^{th}$ power sum in a sequence of variables $x_\lambda=(x_{i\lambda})_{i\geq 1}.$ Let us denote by $\mathcal{A}^k$ the algebra over $\mathbb{C}$ with algebraic basis the elements $P_r(C_\lambda)$
$$\mathcal{A}^k:=\mathbb{C}[P_r(C_\lambda);r\geq 1, \lambda\vdash k].$$
If $\rho=(\rho_1,\rho_2,\cdots, \rho_l)$ is an arbitrary partition and $\lambda\vdash k,$ we define $P_\rho(C_\lambda)$ to be the product of $P_{\rho_i}(C_\lambda),$
$$P_\rho(C_\lambda):=P_{\rho_{1}}(C_\lambda)P_{\rho_2}(C_\lambda)\cdots P_{\rho_l}(C_\lambda).$$
The family $(P_\Lambda)_\Lambda$ indexed by families of partitions, where
$$P_\Lambda:=\prod_{\lambda\vdash k}P_{\Lambda(\lambda)}(C_\lambda),$$
forms a linear basis for $\mathcal{A}^k.$
If we assign degree $r$ to $P_r(C_\lambda),$ then 
$$\mathcal{A}^k=\bigoplus_{n\geq 0}\mathcal{A}_n^k$$
is a graded $\mathbb{C}$-algebra where $\mathcal{A}_n^k$ is the algebra spanned by all $P_\Lambda$ where $|\Lambda|=n.$
The algebra $\mathcal{A}^k$ can be equipped with a hermitian scalar product defined by 
$$<f,g>=\sum_\Lambda f_\Lambda \bar{g}_\Lambda Z_\Lambda$$
for any two elements $f=\sum_\Lambda f_\Lambda P_\Lambda$ and $g=\sum_\Lambda g_\Lambda P_\Lambda$ of $\mathcal{A}^k.$ In particular, we have:
$$<P_\Lambda,P_\Gamma>=\delta_{\Lambda,\Gamma}Z_\Lambda,$$
where $\delta_{\Lambda,\Gamma}$ is the Kronecker symbol. 

If $\gamma$ is a partition of $k$ and $r\geq 1,$ define 
$$P_r(\chi^\gamma):=\sum_{\lambda\vdash k}z_{\lambda}^{-1}\chi^\gamma_\lambda P_r(C_\lambda),$$
where $\chi^\gamma_\lambda$ is the value of the character $\chi^\gamma$ on one element of the conjugacy class $C_\lambda.$ By the orthogonality of the characters of $\mathcal{S}_k,$
$$<\chi^\gamma,\chi^\delta>:=\frac{1}{k!}\sum_{g\in \mathcal{S}_k}\overline{\chi^\gamma(g)}\chi^\delta(g)=\sum_{\rho\vdash k}z_{\rho}^{-1}\chi^\gamma_\rho\chi^\delta_\rho=\delta_{\gamma,\delta},$$
we can write
$$P_r(C_\lambda)=\sum_{\gamma\vdash k}\overline{\chi^\gamma_\lambda}P_r(\chi^\gamma).$$
We may think of $P_r(\chi^\gamma)$ as the $r^{th}$ power sum in a new sequence of variables $y_\gamma=(y_{i\gamma})_{i\geq 1}$ and denote by $s_\rho(\chi^\gamma)$ the schur function $s_\rho$ associated to the partition $\rho$ on the sequence of variables $(y_{i\gamma})_{i\geq 1}.$ Now, for any family of partitions $\Lambda,$ define $$S_\Lambda:=\prod_{\lambda\vdash k}s_{\Lambda(\lambda)}(\chi^\lambda).$$
The family $(S_\Lambda)_\Lambda$ indexed by the families of partitions is an orthonormal basis of $\mathcal{A}^k,$ see \cite{McDo}. 

Let $\gamma$ be a partition of $k$ and consider $V^\gamma$ the irreducible $\mathcal{S}_k$-module associated to $\gamma.$ The group $\mathcal{B}_{kn}^k$ acts on the $n^{th}$ tensor power $T^n(V^\gamma)=V^\gamma\otimes V^\gamma\cdots \otimes V^\gamma$ as follows:
$$\omega. v_1\otimes v_2\otimes\cdots \otimes v_n:=\omega_1v_{p_\omega^{-1}(1)}\otimes \omega_2v_{p_\omega^{-1}(2)}\otimes \cdots \otimes \omega_nv_{p_\omega^{-1}(n)},$$
where $\omega\in ,$ $v_1,v_2,\cdots ,v_n\in V^\gamma$ and $\omega_i\in \mathcal{S}_k$ is the normalized restriction of $\omega$ on the block $p_\omega^{-1}(i).$ If we denote by $\eta_n(\chi^\gamma)$ the character of this representation of $\mathcal{B}_{kn}^k$ then, by \cite[Equation (8.2), page 176]{McDo}, if $\omega\in \mathcal{B}_{kn}^k$ we have:
$$\eta_n(\chi^\gamma)(\omega)=\prod_{\rho\vdash k}({\chi^\gamma_\rho})^{l(\omega(\rho))}.$$

For any partition $\mu$ of $m$ and for any partition $\gamma$ of $k$ define
$$X^\mu(\chi^\gamma):=\det (\eta_{\mu_i-i+j}(\chi^{\gamma})).$$ 
By extending this definition to families of partitions, we obtain the full list of irreducible characters of $\mathcal{B}_{kn}^k.$ If $\Lambda$ is a family of partitions, define 
$$X^\Lambda:=\prod_{\rho\vdash k}X^{\Lambda(\rho)}(\chi^\rho).$$ 

For any two families of partitions $\Lambda$ and $\Delta,$ let us denote by $X^\Lambda_\Delta$ the value of the character $X^\Lambda$ on any of the element of the conjugacy class $C_\Delta.$ By \cite[page 177]{McDo}, we have the following three important identities

$$X^\Lambda_\Delta=<S_\Lambda,P_\Delta>,$$ 
$$S_\Lambda=\sum_{\Gamma}Z_{\Gamma}^{-1}X^\Lambda_\Gamma P_\Gamma$$ 
and
$$P_\Lambda=\sum_{\Gamma}\overline{X^\Gamma_\Lambda} S_\Gamma,$$
where the sums run over families of partitions.

Let $\Lambda$ be a family of partitions and consider the composition $F_\Lambda^k:=\frac{X^\Lambda}{\dim \Lambda}\circ \psi \circ \Proj_{|\Lambda|}$ of morphisms. We would like to see how $F_\Lambda^k$ acts on the basis elements of $\mathcal{I}_\infty^k.$ If $\Delta$ is a family of partitions such that $|\Delta|>|\Lambda|,$ it would be clear then that $F_\Lambda^k(C_\Delta)=0.$ Suppose now that $|\Lambda|\geq |\Delta|,$ we have the following equalities:
\begin{eqnarray}\label{eq_isom}
\big(\frac{X^\Lambda}{\dim \Lambda}\circ \psi \circ \Proj_{|\Lambda|} \big)(\mathbf{C}_{\Delta})&=& \frac{X^\Lambda}{\dim \Lambda} \Big( {|\Lambda|-|\Delta|+m_1(\Delta(1^k))\choose m_1(\Delta(1^k))} \mathbf{C}_{\underline{\Delta}_{|\Lambda|}} \Big) \\
\notag
&=& {|\Lambda|-|\Delta|+m_1(\Delta(1^k))\choose m_1(\Delta(1^k))}\frac{|\Lambda|!(k!)^{|\Lambda|}}{ Z_{\underline{\Delta}_{|\Lambda|}}\dim \Lambda}X^{\Lambda}_{\underline{\Delta}_{|\Lambda|}} \\
\notag
&=& \frac{(k!)^{|\Delta|}}{Z_\Delta}\frac{(|\Lambda|\downharpoonright |\Delta|)}{\dim \Lambda} X^{\Lambda}_{\underline{\Delta}_{|\Lambda|}}.
\end{eqnarray}
This suggests that if we consider the algebra of shifted symmetric functions $\mathcal{A}^{k*}$ isomorphic to $\mathcal{A}^k$ with basis the shifted functions $P^{\#}_\Delta$ defined by 
$$P^{\#}_\Delta:=\prod_{\rho\vdash k}P^{\#}_{\Delta(\rho)}(C_\rho)$$
and we set  
$$P^{\#}_\Delta(\Lambda):=\prod_{\rho\vdash k}P^{\#}_{\Delta(\rho)}(C_\rho)(\Lambda(\rho)),$$
for any family of partitions $\Lambda$ then we obtain the following result.

\begin{theoreme}
The linear map $F^k:\mathcal{I}_\infty^k\longrightarrow \mathcal{A}^{k*}$ defined by $$F^k(\mathbf{C}_{\Delta})=\frac{(k!)^{|\Delta|}}{Z_\Delta}P^{\#}_\Delta$$ is an isomorphism of algebras.
\end{theoreme}

\bibliographystyle{plain}
\bibliography{biblio}
\label{sec:biblio}

\end{document}